\documentclass{amsart}

 \usepackage{amssymb}
\usepackage{amsthm}
\usepackage[backref=page, colorlinks=true, citecolor=cyan, linkcolor=blue]{hyperref}
\usepackage{xcolor}
\usepackage{amsmath}
\usepackage{graphicx}
\usepackage{adjustbox}
\usepackage{bbm}
\usepackage{comment}
\usepackage{inputenc,hyperref}

\newtheorem{theorem}{Theorem}[section]
\newtheorem{step}{Step}
\newtheorem{lem}{Lemma}[section]
\newtheorem{remark}{Remark}[section]
\numberwithin{equation}{section}
\newtheorem*{theorem*}{Theorem}
\newtheorem{rem}{Observation}[section]

\newcommand{\R}{\mathbb{R}}

\newcommand{\C}{\mathcal{C}}

\newcommand{\p}{\partial}
\newcommand{\A}{\mathcal{A}}

\newcommand{\Rd}{{\mathbb{R}^d}}

\newcommand{\Prob}{\mathbb{P}}
\newcommand{\T}{\mathcal{T}}
\newcommand{\F}{\mathcal{F}}

\newcommand{\E}{\mathcal{E}}

\newcommand{\inp}{in}

\begin{document}


%\title{Delayed dynamics of large particle systems}
\title[Interacting particles systems with delay]
{Interacting particles systems with delay and random delay differential equations}

\author[ J.P. Pinasco, M. Rodriguez Cartabia, N. Saintier  ]{
Juan Pablo Pinasco, Mauro Rodriguez Cartabia, Nicolas Saintier}

\address{Juan Pablo Pinasco,
\hfill\break\indent IMAS UBA-CONICET  and  Departamento  de Matem{\'a}tica,
 \hfill\break\indent Facultad de Ciencias Exactas y Naturales, Universidad de Buenos Aires,
 \hfill\break\indent  Av Cantilo s/n, Ciudad Universitaria
 \hfill\break\indent (1428) Buenos Aires, Argentina.}\email{jpinasco@dm.uba.ar }

\address{Mauro Rodriguez Cartabia,
\hfill\break\indent IMAS UBA-CONICET  and  Departamento  de Matem{\'a}tica,
 \hfill\break\indent Facultad de Ciencias Exactas y Naturales, Universidad de Buenos Aires,
 \hfill\break\indent  Av Cantilo s/n, Ciudad Universitaria
 \hfill\break\indent (1428) Buenos Aires, Argentina.}\email{mrodriguezcartabia@gmail.com}

\address{Nicolas Saintier,
\hfill\break\indent IMAS UBA-CONICET  and  Departamento  de Matem{\'a}tica,
 \hfill\break\indent Facultad de Ciencias Exactas y Naturales, Universidad de Buenos Aires,
 \hfill\break\indent  Av Cantilo s/n, Ciudad Universitaria
 \hfill\break\indent (1428) Buenos Aires, Argentina.}\email{nsaintie@dm.uba.ar}

\thanks{Supported by ANPCyT under grant PICT 2016-1022, by CONICET under  grant PIP 11220150100032CO, and by Universidad de
Buenos Aires under grant 2018 20020170100445BA. The authors are members of
CONICET, Argentina.}

\begin{abstract}
In this work we study a kinetic model of active particles with delayed dynamics, and its limit
when the number of particles goes to infinity. This limit turns out to be related to  delayed
differential equations with random initial conditions. We analyze two different dynamics,
one based on the full knowledge of the  individual trajectories of each particle, and another
one based  only on the trace of the particle cloud, loosing track of the individual
trajectories. Notice that in the first dynamic the state of a particles is its path, whereas  it is
simply a point in $\R^d$ in the second case. We analyse in both cases the corresponding
mean-field dynamic obtaining an equation for the time evolution of the  distribution of the
particles states. Well-posedness of the equation is proved by a fixed-point argument. We
conclude the paper with some possible future research directions and modelling
applications.
\end{abstract}


\maketitle

%\tableofcontents

\section{Introduction}

The study of large systems of interacting particles attracted a great deal of attention in the
last years. Among many other topics, we can find a growing literature on flocking, crowds,
traffic, opinion dynamics, wealth distribution, and many other natural and social phenomena
studied using kinetic equations (see e.g. \cite{Bellomo}, \cite{PT}.). 
Here, we will consider delayed dynamics, where the agent
reaction depends on the recent history of other agents behavior,
 in some sense that will be clarified later.

\medskip

Let us  observe that in many cases a delay makes a priori no sense, since nobody try to
convince us of some political position that she/he held on the past, or our reactions in a
middle of a flock (a crowd, a traffic jam) depend only on the present state of the world.
However, the presence of some communication delay, or of a reaction time between the
signal and the response,  are reasonable assumption which immediately lead to introducing a
delay in the model (see for instance \cite{choicucker}, \cite{munz2008delay}). There are many other
contexts in which the recent history is clearly relevant. For instance suppose that many
agents are randomly matched in a two player game, each one is using a mixed strategy, and
they update their strategies after each pairwise encounter depending only on the last game
played, see \cite{pinasco2018game} for an example. On the other hand, using no-regret
algorithms \cite{jafari2001no}, the players try to improve their strategies by considering the
full history of games played.

\medskip

So, let us suppose we have an interacting population of $N$ agents labeled $i=1,\ldots,N$ evolving
  in  $\Rd$. We assume that interactions among agents are the only factor of movement.
In absence of delay we can model the evolution of each agent by the following coupled system:
\begin{equation}\label{classic}
\frac{d}{dt}x^{(i)}(t)
= \frac{1}{N}\sum_{j=1}^N K\left( x^{(i)}(t),x^{(j)}(t)\right),\qquad i=1,\ldots,N,
\end{equation}
where the function $K:\Rd \times \R^d \to \Rd $ models the interaction between any pair of agents.
This is a mean-field framework in the sense that each agent feels the mean influence of the rest of the population
and that interactions are only binary and modelled by the same function $K$ regardless of the particular agents involved in the interaction.
It is a classical setting relevant in various applications ranging from physics to biology and sociology.
In biology for instance we can think of $x^{(i)}$ as the pair (position, velocity) characterizing animal $i$.
The kernel $K$ can model a velocity- alignment mechanism as e.g. in the classical
Cucker-Smale flocking model, or synchronization of oscillators as in the Kuramoto model.
In sociology, $x^{(i)}$ could be the opinion of individual $i$, and $K$ then models
the variation of opinions due to interactions
among individuals (see e.g. \cite{paperopinion} or \cite{toscani2006kinetic} and the references therein).
We refer e.g. to the survey \cite{golsedynamics} for a detailed mathematical introduction to this subject in the absence of delay.

The main purpose of this work is to extend this theory by incorporating a general delay.
To take delay into account in the interaction between any two agents, say $i$ and $j$,
we suppose that agent $i$ reacts to the history
$\{x^{(j)}_t(s):=x^{(j)}(t-s),\, s\in [-\tau,0]\}$ of agent $j$ in the time window $[t-\tau,t]$.
Notice that $x^{(j)}_t \in C([-\tau,0],\R^d)$.
In the mean-field framework this leads to replace \eqref{classic} by
\begin{equation}\label{delayed}
\frac{d}{dt}x^{(i)}(t)
= \frac{1}{N}\sum_{j=1}^N K\left( x^{(i)}(t),x_t^{(j)}\right),\qquad i=1,\ldots,N,
\end{equation}
where $K:\Rd \times C([-\tau,0],\R^d) \to \Rd$. Notice that in the delayed setting the natural
variable to characterize the state of agent $i$ is not the value $x^{(i)}(t)\in \R^d$ at time $t$
but his history $x^{(i)}_t$ in the time window $[t-\tau,t]$.
%, where $x^{(i)}_t\in C([-\tau,0],\R^d)$ is defined as $x^{(i)}_t(s):=x^{(i)}(t-s)$, $s\in [-\tau,0]$. 
Thus the natural
state space is not $\R^d$ but $E:=C([-\tau,0],\R^d)$. From a modelling point of view this
means that
\begin{itemize}
\item[(a)] each agents knows exactly the trajectories $x^{(j)}_t$, $t\ge 0$, of any other agents in the population.
\end{itemize}

It may happen however that some particular choice of kernel $K$ leads to a loss of information in
 the sense that agent $i$ does not react to the precise
trajectories $x_t^{(j)}$ of others agents $j$ in the population but only to the global
distribution of agents at each time $t+s$, $s\in [-\tau,0]$ without caring for the individual
trajectories. Consider e.g. the kernel $K$ given by
\begin{equation}\label{LossMemoryKernel}
 K(x^{(i)}(t),x^{(j)}_t):=
\int_{-\tau}^0 \tilde K\left(x^{(i)}(t),x^{(j)}(t+s)\right) \rho(ds),
\end{equation}
where $\tilde K:\R^d\times\R^d\to \R^d$, and  $\rho(ds)$ is a given probability measure on $[-\tau,0]$. System \eqref{delayed} becomes
\begin{equation}\label{delayed_Loss}
\frac{d}{dt}x^{(i)}(t)
= \int_{-\tau}^0 \frac1N \sum_{j=1}^N \tilde K\left(x^{(i)}(t),x^{(j)}(t+s)\right) \rho(ds)
 ,\qquad i=1,\ldots,N.
\end{equation}
Thus this particular kernel $K$ leads to a situation where information of the individual trajectories is lost so that
\begin{itemize}
\item[(b)]  each agent knows if one agent was located at position $x$ at some time $t+s$,   without knowing exactly which specific agent was there.
\end{itemize}
This means that agents have imperfect memory: they know the trace of the paths followed by others agents, but they are not able to identify which agent travel each path.
A typical situation depicted in Figure \ref{caminos} corresponds to
considering the path in the time window $[t-\tau,t]$ at time $t$ of two agents $x^{(1)}$ and $x^{(2)}$ namely 
$x^{(1)}_t,x^{(2)}_t\in E$. For ease of notation we let $\sigma^{(i)}:=x^{(i)}_t$, $i=1,2$. 
We suppose that these paths intersect ar time $t-h$ for some $h\in [-\tau,0]$ i.e. 
$\sigma^{(1)}(h)=\sigma^{(2)}(h)$, and consider the  paths
\begin{align*}
  \tilde \sigma^{(1)}=\sigma^{(1)}\mathbbm{1}_{[-\tau,h]}+\sigma^{(2)} \mathbbm{1}_{[h,0]},
\qquad  \tilde\sigma^{(2)}=\sigma^{(2)}\mathbbm{1}_{[-\tau,h]}+\sigma^{(1)} \mathbbm{1}_{[h,0]}.
\end{align*}
Then it is easily seen that for any $x\in\R^d$,
\begin{align*}
  \int_{-\tau}^0  K(x,\sigma^{(1)}(s))+ K(x,\sigma^{(2)}(s))\, d\rho(s)
 =\int_{-\tau}^0  K(x,\tilde\sigma^{(1)}(s))+K(x,\tilde\sigma^{(2)}(s))\, d\rho(s).
\end{align*}
Thus an agent located at $x$ cannot distinguish the precise individual trajectories but only the trails they left without knowing who is where.

This setting is well known in other problems in partial differential equations.
We mention for instance materials with memory as studied by Dafermos \cite{dafermos}.
As
observed in \cite{FabGP}, these kind of problems does not have the property of minimality,
since different initial data generate different solutions.

\begin{figure}
  \begin{minipage}{0.3\textwidth}
  \fbox{\adjustbox{trim={.051\width} {0.2\height} {0.01\width} {0.2\height},clip}
    {  \includegraphics[scale=0.15]{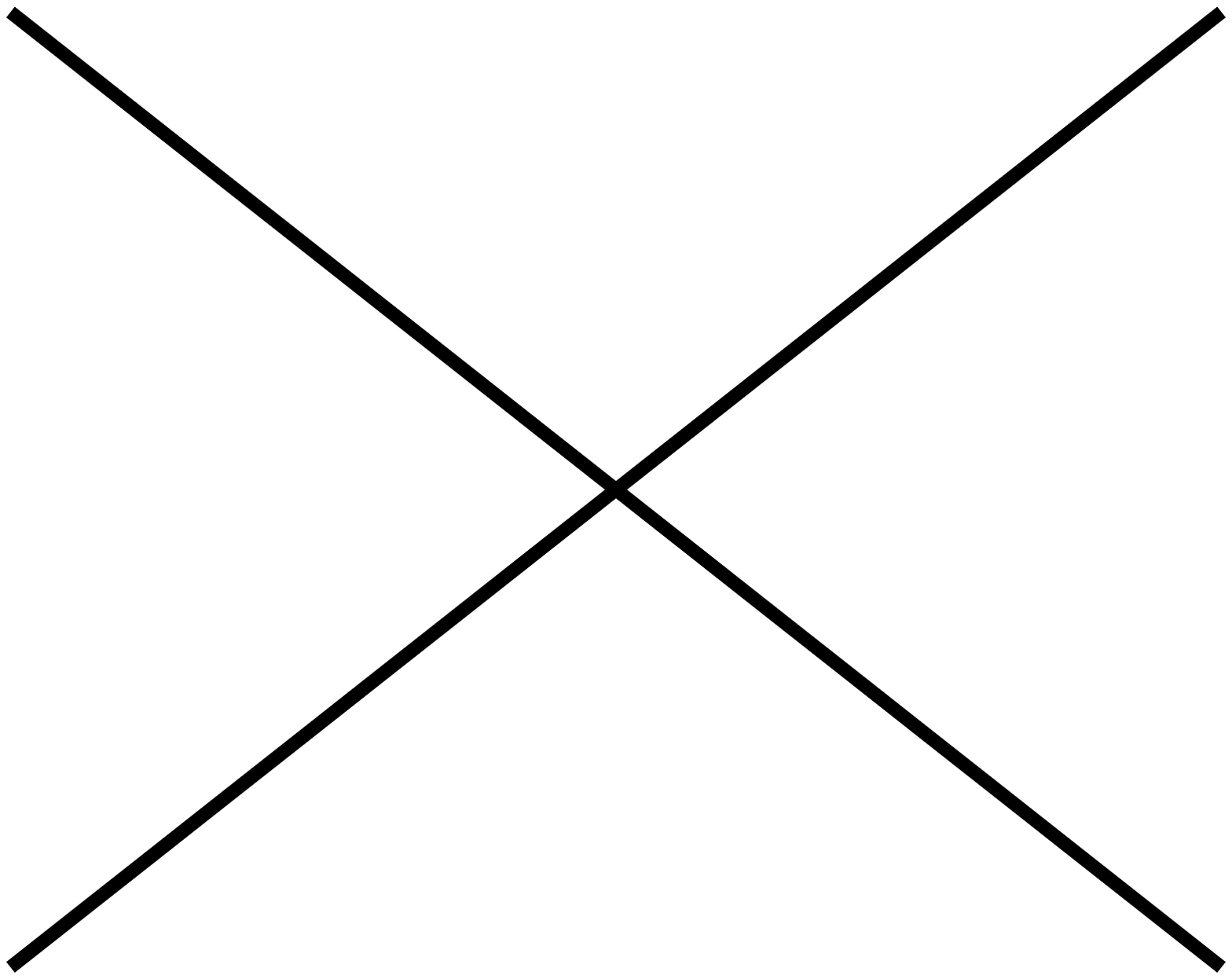}}}
  \end{minipage}
  \begin{minipage}{0.3\textwidth}
  \fbox{\adjustbox{trim={.051\width} {0.2\height} {0.01\width} {0.2\height},clip}
    {  \includegraphics[scale=0.15]{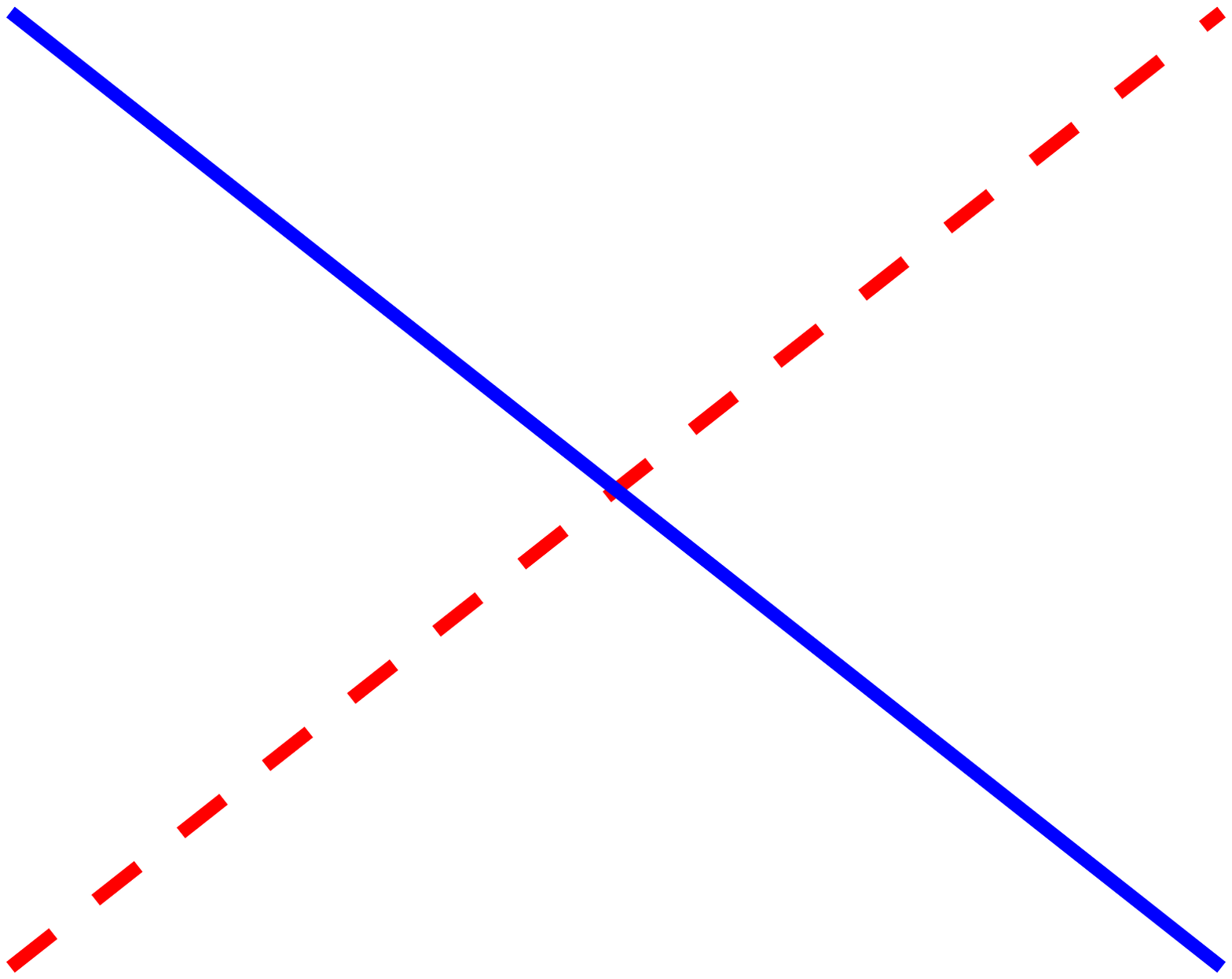}}}
  \end{minipage}
  \begin{minipage}{0.3\textwidth}
  \fbox{\adjustbox{trim={.051\width} {0.2\height} {0.01\width} {0.2\height},clip}
    {  \includegraphics[scale=0.15]{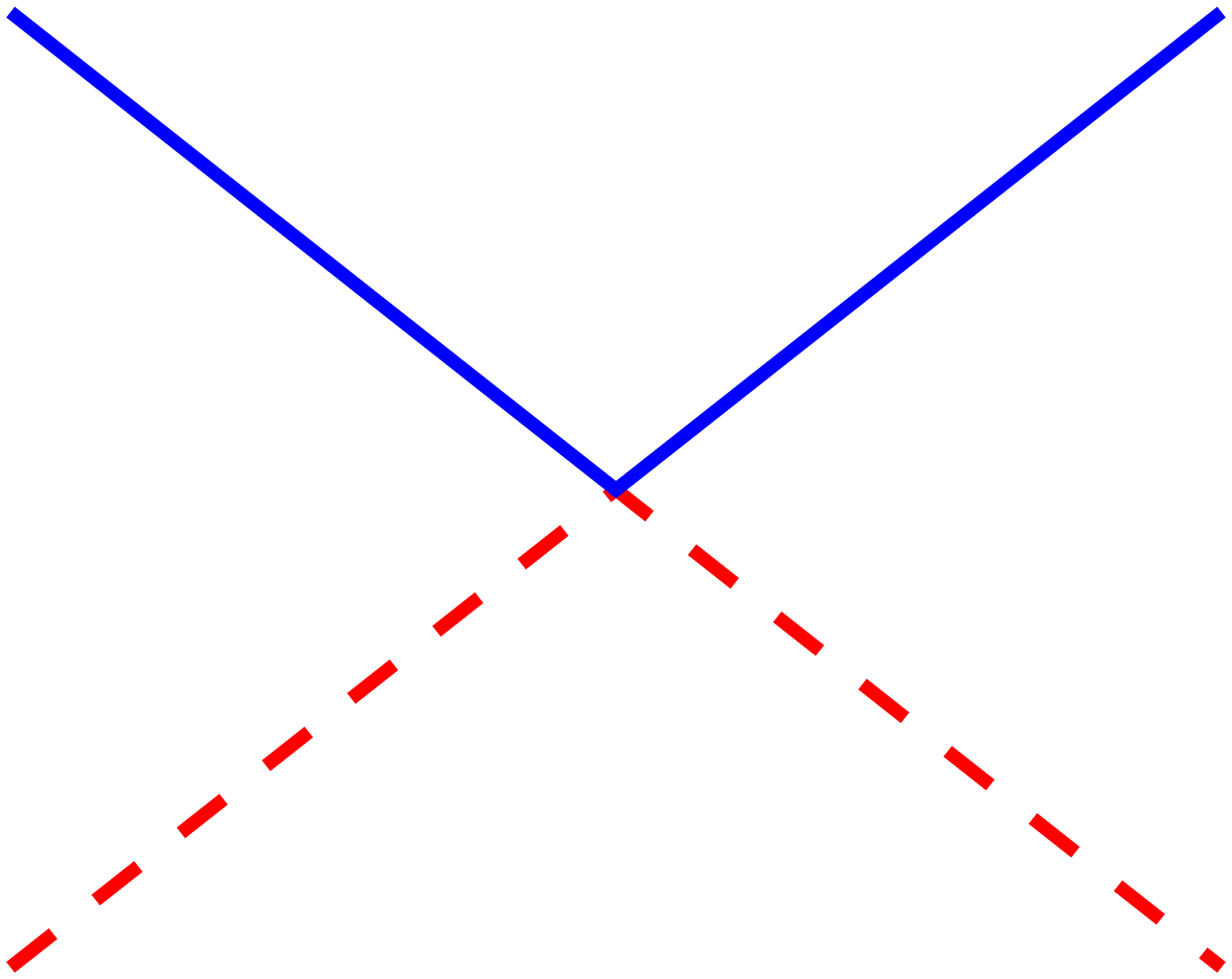}}}
  \end{minipage}
\caption{In the first model, a player who receives the information of the image at the left panel
can not distinguish between two players who crossed
(central panel) or two who touched and bounced (right panel).}\label{caminos}
\end{figure}

\medskip

One of the main concern in dealing with interacting particles system like \eqref{classic} or
\eqref{delayed} is the determination  of its limit as $N\to +\infty$. The treatment of this
question in absence of delay, namely for \eqref{classic}, is classic (see e.g.
\cite{golsedynamics}) and relies on rewriting system \eqref{classic} as a single equation for
the empirical measure $\mu^N(t):=\frac1N \sum_{i=1}^N \delta\{x^{(i)}(t)\}$, where
$\delta\{x^{(i)}(t)\}$ is the Dirac mass at $x^{(i)}(t)$. Notice that $\mu^N(t)$ belongs to
$\Prob(\R^d)$, the set of probability measure on $\R^d$. It is indeed easy to see that
$\mu^N(t)$ is the unique solution to
\begin{equation}\label{KineticEquClassic}
\frac{d}{d t}\mu  + div(\F[\mu]\, \mu)=0,
\end{equation}
with initial condition $\mu^N(0)$. Here, for a given probability measure $\mu$ on $\R^d$,
$\F[\mu]$ is the vector field given by $ \F[\mu](x)=\int_{\R^d} K(x,y)\mu(dy)$.
Well-posedness in $C([0,+\infty),\Prob(\R^d))$ of \eqref{KineticEquClassic} for a wide class
of  initial condition is well-known (see e.g. \cite{golsedynamics} and references therein) 
assuming that $K$ is
bounded Lipschitz. In particular it is known that the solution depends continuously on the
initial condition. Thus if $\mu^N(0)\to \mu(0)$ as $N\to +\infty$,  we have convergence of
the corresponding solutions namely  $\mu^N(t)\to \mu(t)$. In that sense we can consider
\eqref{KineticEquClassic} as the correct equation to treat the interacting population 
whether  the number of agents is finite or not.  
We mention that measure-valued solutions are proving increasingly useful to model various phenomena 
ranging from biology (see e.g. flocking e.g. \cite{CCR}, population dynamic e.g.  
\cite{AS},\cite{CCC},\cite{CCGU}), 
game theory (e.g. \cite{AFMS},\cite{PRCS}), 
sociology (opinion formation process  \cite{PPS},\cite{PPSS}, political polarizaition \cite{PSV},\cite{SPV}) 
to traffic (see e.g. \cite{EBA}).

In this paper  we extend this approach to the  delayed system \eqref{delayed}. To do so we
notice that the solution $\mu(t)$ of \eqref{KineticEquClassic} can also be thought of as the
law of the solution to the ODE $x'(t)=\F[\mu(t)](x(t))$ when the initial condition is chosen
randomly according to the probability distribution $\mu(0)\in \Prob(\R^d)$. In the delayed
setting we are thus led to consider delay differential equation with random initial condition.
We prove the well-posedness of such equations and argue that they are  the correct limit of
the system \eqref{delayed} as $N\to +\infty$. Notice that now the state of the population at
time $t$ is given by a probability measure $\mu(t)$ on $E$ in case (a) or on $\R^d$ in case
(b). In any case
\begin{itemize}
\item[(c)] at any time $t\ge 0$, there exists  a family of probability distributions
    $\mu(t+s)$, $s\in [-\tau,0]$, giving the density of agents in the time-window $[t-\tau,t]$ which is common knowledge for all agents.
\end{itemize}

\begin{remark}
Let us mention briefly that in $(c)$ we are using the notion of {\it common knowledge}
introduced by Aumann, see for instance \cite{koessler2000common}, which implies that
everybody knows $\mu(x,t+s)$, and also   everybody knows that everybody knows that, and
so on. So, everybody will use this family of measures in order to predict the evolution of the
system. Let us note that if a particular agent knows exactly the position of finitely many
agents, this information does not enter in her/his analysis, since it is a set of zero measure
respect to $\mu$.
\end{remark}

\bigskip

As an application of our model, we can consider a population of ants moving on some planar
domain. Usually, they leave traces of pheromones that serve as a guide to other ants, and
the concentration decay with time. This chemical concentration in some ball around an ant
can be easily modeled with a kernel supported in this ball, and  weighting   different the
paths at different times, for instance, using
$$
\int_{-\tau}^0 \sum_{j=1}^N K\left(x^{(i)}(t),x^{(j)}(t+s)\right) \rho(s)ds
$$
with $\rho\to 0$ when $t\to -\tau$, being $\tau$ some characteristic time for the duration of
the pheromones. Let us mention the pioneer work of Fontelos and  Friedman  \cite{ants},
where the behavior of ants was modeled using
a system of two partial differential equations, one governing the distribution of ants, and the other
one devoted to the generation and decay of the pheromones, see
also \cite{ants2}. Since new pheromones appear in the path each ant travelled, it is enough to know their
paths and weighting them with a function of $t$ in order to model the decay of the chemical trail.  We will
study this problem in a separate work.

\medskip

The paper is organized as follow. We first recall some preliminary results about probability
measures and differential equations, with and without delay. We then study  interacting
particles system like \eqref{delayed} and its limit as $N\to +\infty$ in the form of delayed
differential equations with random initial condition. The main result of this paper is the
well-posedness of such equation. We go on examining in details the case of a kernel like
\eqref{LossMemoryKernel}. We conclude by examining some possible direction for future
research. The proofs are given at the very end of the paper for an easy flow.

\section{Preliminaries}\label{preliminares}

In this section we establish some notations and recall known results  that we use throughout this article.

Let us fix some constants  $\tau>0$, the maximum value of the delay. 
We denote $x^{(i)}_t$ the path in the time-window $[t-\tau,t]$ of a an agent $i$. 
It is an element of the space
\begin{equation}\label{espacioC}
  E:= C([-\tau,0],\Rd). 
\end{equation}
with $x^{(i)}_t(s):=x^{(i)}(t-s)$, $s\in [-\tau,0]$. 
We will usually denote $\sigma$ a generic element of $E$. 
We endow $E$ with the usual sup-norm $\|\sigma\|_\infty = \max_{s\in [-\tau,0]} |\sigma(s)|$, $\sigma\in E$.

\subsection{Preliminaries on probability measures}

Let $(X,d)$ be a metric space.
We denote $\Prob(X)$ the set of Borel probability measures on $(X,d)$.
For instance $\delta\{x\}$ is the Dirac measure at the point $x\in X$ which is defined
for any $A\subset X$ Borel by $\delta\{x\}(A) = 1$ if $x\in A$, and $\delta\{x\}(A)=0$ if $x\notin  A$.
Notice in particular that
\begin{equation}\label{deltadedirac}
  \int_X\varphi(y)\, d \delta\{x\} (y)=  \varphi(x)
\end{equation}
for any continuous function $\varphi:X\to \R$.

A measure $\mu\in \Prob(X)$ has finite first moment if
$\int_X d(x,x_0)\,d\mu(x)<\infty$ for some (hence for any) $x_0\in X$.
We denote $\Prob_1(X)$ the set of such probability measures.

As a matter of notation we will indifferently denote $\int_X f(x)\,d\mu(x)$ or $\int
f(x)\mu(dx)$.

Given another measure space $Y$ with some $\sigma$-algebra $\Sigma_Y$,
the push forward of a measure $\mu$ on $X$ by a measurable
function $f:X\to Y$ is the measure $f\#\mu$ on $Y$ defined by
\begin{equation}\label{pushforward}
  f\#\mu (A):=\mu(f^{-1}(A))
\end{equation}
for each $A\in \Sigma_Y$. This is equivalent to saying that
$$ \int_Y \phi \,d(f\#\mu) = \int_X \phi\circ f\,d\mu $$
for any bounded measurable $\phi:Y\to \R$.

On $\Prob_1(X)$ we consider the Monge-Kantorovich distance $W_1$ defined by
\begin{equation}\label{DefW1}
 W_1\left(\mu,\tilde\mu\right):=
\sup \left\{ \int_X\varphi\,d\left(\mu-\tilde\mu\right):\, \varphi:X\to \R, \text{and } Lip(\varphi)\leq 1\right\}
\end{equation}
for $\mu,\tilde \mu\in \Prob_1(X)$.
 Here  $Lip(\varphi)$ denotes  the Lipschitz constant of $\varphi$ i. e. the least constant
$C>0$ such that $|\phi(x)-\phi(y)|\le Cd(x,y)$ for any $x,y\in X$. It is known that if $(X,d)$ is
complete (resp. Polish, compact) then so is $(\Prob_1(X),W_1)$. Moreover given
$\mu_k,\mu\in \Prob_1(X)$, the convergence $W_1(\mu_k,\mu)\stackrel{k\to
+\infty}{\longrightarrow}0$ is equivalent to the convergence $\int \phi\,d\mu_k\to \int
\phi\,d\mu$ for any continuous function $\phi:X\to\R$ with at most linear growth (i. e.
$|\phi(x)|\le C(1+|x|)$ for some $C>0$). We refer e. g. to \cite{villanioptimal} for the proof of
these statements and more details on Monge-Kantorovich distance.

When $X=\R^d$,  we will need to consider space
\begin{equation}\label{espacioE}
  \E:=C\left([-\tau,0],\mathbb{P}_1(\Rd)\right)
\end{equation}
that we endow with the  sup distance $\mathcal{W}_1$ defined by
\begin{equation}\label{normaM}
  \mathcal{W}_1(\mu,\nu):=\max_{s\in[-\tau,0]}W_1\left(\mu(s),\nu(s)\right)
\qquad \mu,\nu\in \E.
\end{equation}
Notice that if $(X,d)$ is complete, so that $(\mathbb{P}_1(\Rd),W_1)$ is complete,
then $(\E,\mathcal{W}_1)$ is complete.
Eventually given some $\mu\in C\left([-\tau,T],\mathbb{P}_1(\Rd)\right)$
and $t\in[0,T]$ we let $\mu_t\in \E$ be defined as $\mu_t (s):=\mu( t+s)$, $s\in [-\tau ,0]$.

\subsection{Preliminaries on Ordinary Differential Equations}

Consider a  vector field  $\F:(x,t)\in\Rd\times [0,T]\to\F(x,t)\in\Rd$ that we suppose
continuous in $(x,t)$ and also globally  Lipschitz in $x$ uniformly in $t$. Then it is well-known
that for any $x\in \R^d$ and $s\in \R$, the differential equation
\begin{equation}\label{EDO}
 y'(t)=\F(y(t),t) \qquad t\in \R
\end{equation}
with initial condition $y(s)=x$, has a unique global solution that we denote  $\T(s,t,x)$.
We then have the flow property $\T(s,t+t',x)= \T(t,t+t',\T(s,t,x))$.
Eventually for any $s,t\in\R$,
$\T(s,t,\cdot):\R^d\to \R^d$ is a $C^1$-diffeomorphism with inverse $\T(t,s,\cdot)$.
For simplicity we will denote $\T(t,\cdot):=\T(0,t,\cdot)$, $t\in\R$.

Take $s=0$ and suppose now that the initial condition $x$ in \eqref{EDO}  is chosen at random following some probability measure $\mu(0)\in \Prob(\R^d)$.
Then the solution $y$ to \eqref{EDO} is a random vector whose distribution $\mu(t)$ at time $t$ is the probability measure $\mu(t):=\T(t,\cdot)\sharp\mu(0)$.
It is well-known that for any $T>0$, $\mu$ is the unique solution in $C([0,T],\Prob(\R^d))$
to the first order equation
\begin{equation}\label{Transport}
 \p_t\mu + \text{div}(\F\mu)=0
\end{equation}
in the sense that
\begin{equation}\label{DefTransport}
 \int_{\R^d} \phi\,d\mu(t) = \int_{\R^d} \phi\,d\mu(0)
 + \int_0^t\int_{\R^d} \F(x,t)\nabla\phi(x)\,\mu(s)(dx)ds
 \end{equation}
for any $\phi\in C^1_c(\R^d)$.
 We refer e.g. to \cite{villanioptimal}. Moreover, given
two initial condition $\mu(0),\tilde\mu(0)\in \Prob_1(\R^d)$, the corresponding solutions
$\mu(t),\tilde\mu(t)$ to \eqref{Transport} satisfy
\begin{equation}\label{Continuity}
W_1(\mu(t),\tilde\mu(t))\le e^{Lt} W_1(\mu(0),\tilde\mu(0)) \qquad t\in\R.
\end{equation}
where $L>0$ is such that $|\F(x,t)-\F(y,t)|\le L|x-y|$, $x,y\in\R^d$, $t\ge 0$.
This follows easily from \eqref{DefW1} and $\mu(t):=\T(t,\cdot)\sharp\mu(0)$,
$\tilde\mu(t):=\T(t,\cdot)\sharp\tilde\mu(0)$.

\medskip

To end this section we recall the classical  Gr\"onwall's inequality:

\begin{lem}[Gr\"onwall's inequality]\label{Gronwall1}
Let be $u,a,b: [a,b]\to [0,\infty)$  continuous and such that
$$u(t)\le a(t)+b(t)\int^t_0 u(h)\,dh \qquad \text{for any $t\ge 0$.}$$
Then
$$ u(t)\le a(t)+b(t)\int_0^t a(h) e^{\int_h^t b}\,dh
\qquad \text{for any $t\ge 0$.}$$
\end{lem}

\subsection{Preliminaries on Delay Differential Equations}

A Delay Differential Equation (DDE) is an equation of the form
\begin{equation}\label{DDE}
 x'(t)=F(t,x_t)
\end{equation}
where $x$ takes in values in $\R^p$ for a given $p$, $x_t\in C([-\tau,0],\R^p)$ is defined as before by $x_t(s):=x(t+s)$, and
$F:(t,\sigma)\in \R\times C([-\tau,0],\R^p)\to F(t,\sigma)\in\R^p$.
Thus the evolution of $x$ at time $t$ depends on the past history of $x$ in $[t-\tau,t]$.
For instance if $F(t,\sigma)=\tilde F(t,\sigma(0))$ for some $\tilde F:\R\times \R^p\to\R$,
then the DDE \eqref{DDE} becomes the ODE $x'(t)=\tilde F(t,x(t))$.
In general given a probability measure $\rho\in \Prob([-\tau,0])$, we can consider
 $F(\sigma)=\int_{-\tau}^0 \tilde F(s,\sigma(s))\,d\rho(s)$.
In that case the DDE \eqref{DDE} becomes
$$ x'(t) = \int_{-\tau}^0 \tilde F(s,x(t+s))\,d\rho(s) $$
so that the evolution of $x$ at time $t$ depends on the previous states $x(t+s)$ weighted by
$\rho$. We recover the ODE $x'(t)=\tilde F(x(t),t)$ taking $\rho=\delta\{0\}$.

To start the evolution from $t=0$, we must  prescribe the values of $x(s)$,
$s\in [-\tau,0]$. Thus the DDE must be complemented with an initial condition of the form
$$ x(s)=\phi(s),\qquad s\in [-\tau,0] $$
i. e. $x_0=\phi$ in $[-\tau,0]$, for some function $\phi\in C([-\tau,0],\R^N)$.

We refer to the classical book \cite{hale} and also to the more recent book
\cite{smith2011introduction} for a thorough treatment of DDE. We quote in particular the
following  existence and uniqueness result from \cite{smith2011introduction}, see also
section 2.3 in \cite{hale}:

\begin{theorem}\label{smith}[Thm. 3.7, p.32]
Suppose that $F:\in C([-\tau,0],\R^p)\times \R\to  \R^p$ is continuous and satisfies the
following local Lipschitz condition: for all $a,\,b\in\R$ and $M>0$, there exists $L>0$ such
that:
$$\left|F(t,\sigma)-F(t,\tilde\sigma)\right|\leq L\|\sigma-\tilde\sigma\| $$
for any $t\in [a,b]$ and $\sigma,\tilde\sigma\in C([-\tau,0],\R^p)$ such that
 $\|\sigma\|,\|\tilde\sigma\|\leq M$.

Then for any $M>0$, there exists $\delta>0$ depending only on $M$, such that
for any initial condition $\phi\in C([-\tau,0],\R^p)$ such that  $\|\phi\|\leq M$,
there exists a unique solution $x$ to the DDE defined up to time $\delta$:
\begin{equation*}
\left\{
\begin{array}{rl}
  x'(t) &=F(x_t,t) \qquad 0\le  t< \delta, \\
  x_0&=\phi.
\end{array}
\right.
\end{equation*}
\end{theorem}

\begin{rem}\label{Global}
It follows from Remark 3.8 in \cite{smith2011introduction} that in the setting of Theorem \ref{smith},
if $F$ satisfies a global Lipschitz condition instead of a local one, that is,
 if $L$ can be chosen independent of $a$, $b$, $M$, then the solution $x$ exists for all $t\geq 0$.
\end{rem}

\medskip

We will see in the next section how the study of interacting particles systems with delay naturally leads to consider
 DDE with random initial condition
in the same spirit that \eqref{Transport} stems from considering the ODE \eqref{EDO}
where the initial condition is chosen at random following some probability distribution $\mu(0)$.

\section{Interacting particles system with delay.} \label{memoriaimperfecta}

As explained in the introduction this paper is devoted to the study of a interacting particles
system with delay of the form \eqref{delayed}, namely
\begin{equation}\label{MainEqDelay}
\frac{d}{dt} x^{(i)}(t)
= \frac{1}{N}\sum_{j=1}^N K\left(x^{(i)}(t),x_t^{(j)}\right),\qquad i=1,\ldots,N,
\end{equation}
where the interaction kernel is  $K:\R^d \times E\to \R^d$ with $E=C([-\tau,0],\R^d)$.

We assume that
\begin{enumerate}
\item[(H)] $K:(x,\sigma)\in\R^d \times E\to K(x,\sigma)\in\R^d$ is globally Lipschitz: there
    exists $L>0$ such that for any $x,\tilde x\in\R^d$ and any $\sigma,\tilde\sigma\in E$,
$$ |K(x,\sigma) - K(\tilde x,\tilde\sigma)|\le L(|x-\tilde x|+|\sigma-\tilde\sigma|). $$

\end{enumerate}
A direct application of Theorem \ref{smith} and Remark \ref{Global} shows that

\begin{theorem}
If $K$ satisfies asssumption (H) then for any $\sigma_{in}^1,\ldots,\sigma_{in}^N\in E$,
there exists a unique solution to \eqref{MainEqDelay} with initial condition
$x^i_0=\sigma_{in}^i$, $i=1,\ldots,N$.
\end{theorem}

\begin{proof}
For each $i=1,\ldots,N$, define $F_i:C([-\tau,0],(\R^d)^N)\to \R^d$
by
$$ F_i(\sigma) = \frac{1}{N}\sum_{j=1}^N K(\sigma^i(0),\sigma^j),
\qquad \sigma=(\sigma^1,\ldots,\sigma^N).  $$
Then \eqref{MainEqDelay} is $\frac{d}{dt}\sigma^i(t)=F_i(\sigma^1_t,\ldots,\sigma^N_t)$,
$i=1,\ldots,N$.
It is easily seen using assumption (H) that $F_1,\ldots,F_N$ are globally Lipschitz.
The result then follows from Theorem \ref{smith} and Remark \ref{Global}.
\end{proof}

Denote $\mu^N(t)$ the empirical measure associated to $x^1_t,\ldots,x^N_t$ namely
\begin{equation}\label{empirica}
  \mu^N (t):=\frac{1}{N}\sum_{i=1}^N\delta\left\{x^{(i)}_t\right\},
\end{equation}
where $\delta\left\{x^{(i)}_t\right\}$ is the Dirac mass at the point $x^i_t$ of $E$.
Notice that $\mu^N\in C([0,+\infty),\Prob(E))$.
Independently for any $\mu\in  C([0,+\infty),\Prob(E))$, consider the vector fields
$\F[\mu]:\R\times \R^d\to \R^d$ defined by
\begin{equation}\label{VF}
 \F[\mu](t,x) = \int_E K(x,\sigma)\,  \mu(t)(d\sigma).
\end{equation}
Then \eqref{MainEqDelay} can be written as
\begin{equation}\label{MainEqDelay30}
\frac{d}{dt} x^{(i)}(t) = \F[\mu^N](t,x^{(i)}(t))   \qquad i=1,\ldots,N.
\end{equation}

We can interpret this system of equations thinking that we are solving the equation
$$ \frac{d}{dt} x(t) = \F[\mu^N](t,x(t))  $$
choosing the inital condition in the set $\{x^{(1)}_0,\ldots,x^{(N)}_0\}$ with equiprobability i.
e. following $\mu^N(0)$. Then $\mu^N(t)$ is the distribution of $x_t$. Thus as $N\to +\infty$
the limit $\mu(t):=\lim_{N\to +\infty}\mu^N(t)$, if it exists, should satisfy the equation
\begin{eqnarray}\label{LimitEq}
\frac{d}{dt} X(t) = \F[\mu](t,X(t))
\end{eqnarray}
where $X_t$ is a random variable with values in $E$ whose distribution is $\mu(t)$, and the
initial condition is a given random variable $X_0:\Omega\to E$ (where $\Omega$ is a
underlying probability space) with distribution $\mu(0):=\lim_{N\to +\infty}\mu^N(0)$ (if the
limit exists). This equation is understood in a path-wise sense: for almost any $\omega\in
\Omega$,
\begin{equation}\label{LimitEq2}
\begin{split}
& X(t)(\omega) = X_0(0)(\omega) + \int_0^t\F[\mu](s,X(s)(\omega))\,ds \qquad t\ge 0,\\
& X(t)(\omega) = X_0(t)(\omega) \qquad t\in [-\tau,0].
\end{split}
\end{equation}

We are thus led to consider the random delay differential equation
\begin{equation}\label{MainRDDE}
\left\{
\begin{array}{rl}
  X'(t) &= \F[\mu](t,X(t))=\int_\C K(X(t),\sigma)\,  d\mu(t,\sigma),\, \mu(t) = \mathcal{L}(X_t), \\
  \mu(0) & = \mathcal{L}(X_0),
\end{array}
\right.
\end{equation}
where $\mathcal{L}(X_t)\in Prob(E)$ denotes the distribution of $X_t$.
We  have the following well-posedness result:

\begin{theorem}\label{GeneralThm}
Suppose that $K:\R^d \times E\to \R^d$ satisfies assumption (H) above. For any $R_0>0$
and any initial distribution $\mu(0)\in Prob(E)$  supported in the ball $B_0(R_0)\subset E$,
there exists a unique solution to \eqref{MainRDDE} in a path-wise sense.

Moreover, given two initial condition $\mu_{in},\nu_{in}\in \Prob(E)$ supported in some ball
$B_0(R_0)$, the law $\mu(t)$ and $\nu(t)$ of the corresponding solution satisfy
\begin{equation}\label{ContDep2}
  W_1(\mu( t),\nu(t))\leq  r(t) W_1(\mu_{\inp},\nu_{\inp})
\end{equation}
for some explicit continuous function $r:[0,+\infty)\to [0,+\infty)$ such that $r(0)=1$.
\end{theorem}

The proof of this result is given in section \ref{SectionProofGeneralThm} below. It consists
first in rewriting the problem as fixed-point equation for curves in $\Prob(E)$, and then
applying the standard Banach fixed-point Theorem. In particular the law $\mu(t)$ of $X_t$ is
shown to be characterized by the fixed-point equation
\begin{equation}\label{FP5}
 (res[t]\circ ext[\mu])\sharp\mu_{in} = \mu(t) \qquad t\ge 0,
\end{equation}
where the restriction operator $res[t]$, $t\ge 0$, is defined by
$$
\begin{array}{rccl}
  res[t]:&C([-\tau, +\infty),\mathbb{R}^d)&\rightarrow & E \\
   &\sigma&\mapsto & \sigma_t,
\end{array}
$$
and  the extension operator $ext[\mu]$ is
\begin{equation}\label{definicionSt}
\begin{array}{rccl}
  ext[\mu]:& E &\rightarrow & C([-\tau,+\infty),\R^d)\\
  &\sigma &\mapsto &\left\{ \begin{array}{ll}
    \sigma(t) &\text{ if } t\leq 0 \\
    \T(t,\sigma(0);\mu)&\text{ if } t> 0,
  \end{array}
  \right.
\end{array}
\end{equation}
being $\T(s,t,x;\mu)$ the flow of te vector field $(t,x)\to \F[\mu](t,x)$.

Coming back to the system \eqref{MainEqDelay}, suppose that the empirical measure
associated to the initial conditions $x^{(i)}_0$, namey $\mu^N(0) = \frac1N \sum_{i=1}^N
\delta\{x^{(i)}_0\}\in Prob(E)$ converges as $N\to +\infty$ to some measure
$\mu(0)\in\Prob(E)$. Denote $\mu(t)$ the solution of the fixed-point equation \eqref{FP5} i.
e. $\mu(t)$ is the law of the unique solution $X(t)$ to \eqref{MainRDDE}. Then according to
\eqref{ContDep2},
$$ W_1(\mu^N(t),\mu(t))\leq  r(t) W_1(\mu^N(0),\mu(0)) \qquad t\ge 0, $$
so that for any $t\ge 0$,
$$ \lim_{N\to +\infty}W_1(\mu^N(t),\mu(t))=0. $$
We can thus consider that \eqref{FP5}, or equivalently \eqref{MainRDDE}, is the correct
equation to describe an arbitrary population of  agents interacting through binary interaction
via the kernel $K$.

\section{Imperfect vs perfect memory}

In our general model an agent at some time $t$  updates his trajectory reacting to the past
history in $[t-\tau,t]$ of all the others individuals. We can thus consider he/she has a perfect
memory of all individual history in the recent past $[t-\tau,t]$. However, it may happen that
for some specific interaction kernel $K$ this complete knowledge is not fully used. Consider
for instance the kernel
\begin{equation}\label{DefKImperfect}
 K(x,\sigma) = \int_{-\tau}^0 \tilde K(x,\sigma(s))\,d\rho(s) \qquad x\in\R^d,\,
\sigma\in E,
\end{equation}
for some probability measure $\rho\in \Prob([-\tau,0])$ and kernel
$\tilde K:\R^d\times \R^d\to \R^d$.
Then the system \eqref{MainEqDelay} of interacting agents becomes
\begin{eqnarray}\label{Eq10}
\frac{d}{dt} x^{(i)}(t)
= \frac1N \sum_j \int_{-\tau}^0 \tilde K(x^{(i)}(t),x^{(j)}(t+s))\,d\rho(s).
\end{eqnarray}

We assume that
\begin{enumerate}
\item[(H')] $\tilde K:\R^d \times \R^d\to \in\R^d$ is globally Lipschitz: there exists $L>0$ such that for any $x,\bar x,y,\bar y\in\R^d$,
$$ |K(x,y) - K(\bar x,\bar y)|\le L(|x-\bar x|+|y-\bar y|). $$
\end{enumerate}
It follows that $K$ satisfies assumption (H)  so that \eqref{Eq10}
has a unique solution for any given initial condition $x^{(1)}_0,\ldots,x^{(N)}_0$.

Introducing
$$ \tilde \mu^N(t) = \frac1N \sum_{j=1}^N \delta\{x^{(j)}(t)\}
\in \Prob(\R^d), \qquad s\in \R, $$
equation \eqref{Eq10} can be rewritten
\begin{equation}\label{MainEqDelayImperfect}
\frac{d}{dt} x^{(i)}(t)
 =  \int_{-\tau}^0 \Big(\int_{\R^d}\tilde K(x^{(i)}(t),y)\,\tilde\mu^N(t+s)(dy)\Big)\,d\rho(s)
\qquad i=1,\ldots,N.
\end{equation}
The measure $\tilde \mu^N(s)$ gives the distribution at time $s$ of the agents in $\R^d$.
Thus the time evolution of each $x^{(i)}(t)$ depends only on  this distribution at time $t+s$,
$s\in [-\tau,0]$, and not on the full knowledge of  individual trajectories. In that sense we
can say that agents have imperfect memory of the past: they only need to know the past
distribution of the whole population and not the precise trajectory of each individual. So it
seems that in this case the natural state space is $\Prob(\R^d)$.

For any $\tilde \mu\in C([-\tau,+\infty],\Prob(\R^d))$, let us introduce the vector field
\begin{equation}\label{VF3}
 \tilde\F[\tilde \mu](t,x) :=
\int_{-\tau}^0 \Big(\int_{\R^d}\tilde K(x,y) \,\tilde\mu(t+s)(dy)\Big)\,d\rho(s)
\qquad t\ge 0,\,x\in\R^d.
\end{equation}
Thus,
$$ \frac{d}{dt} x^{(i)}(t) = \tilde\F[\tilde \mu^N](x^{(i)}(t)) \qquad i=1,\ldots,N,\quad t\ge 0. $$

Given some test function   $\varphi\in C^1_c(\Rd)$, taking the time derivative of
$\int \phi\,d\mu^N(t)=\frac1N \sum_{i=1}^N \varphi(x^{(i)}(t))$, we obtain
\begin{align*}
\frac{d}{dt} \int_{\Rd} \varphi(x)\,d\mu^{N}(x,t)
& =\frac{1}{N}\sum_{i=1}^N  \nabla \varphi(x^{(i)}(t))  \cdot \tilde\F\left[\mu^N\right](x^{(i)}(t)) \\
&= \int_{\Rd} \nabla \varphi(x)  \cdot \tilde\F\left[\mu^N\right](x)\,\mu^{N}(t)(dx).
\end{align*}
Thus $\mu^N$ is a solution, in the sense of definition \eqref{DefTransport}, of the equation
\begin{equation}\label{MFEqu}
 \frac{d}{dt} \mu   + \text{div}\left( \tilde\F[\mu]\, \mu  \right) = 0,
\end{equation}
with initial condition  $\mu_0=\tilde \mu^N_0$ i. e. $\mu(s)=\tilde \mu^N(s)$ for $s\in
[-\tau,0]$.

The next result states this equation is well-posed for arbitrary initial condition:

\begin{theorem} \label{TeoExist}
For any $R_0>0$ and any initial condition  $\mu_{in}\in\E$ such that
$\mu_{in}(t)$ is suported in the ball $B_0(R_0)\subset \R^d$ for any $t\in [-\tau,0]$,
there exists a unique solution
$\mu\in  C\left([-\tau,\infty),\mathbb{P}_1\left(\Rd\right)\right)$
to \eqref{MFEqu} with intial condition $\mu_0=\mu_{in}$.
Moreover $\mu(t)$ has compact support for any $t\ge 0$.

Eventually, the solution  depends  continuously  on  the initial data: there exists a continuous increasing function
$r:[0,\infty)\to [1,\infty)$ with  $r(0)=1$ such that for any compactly supported
initial conditions  $\mu_{in}$ and $\nu_{in}$, the corresponding solutions $\mu(t)$ and $\nu(t)$ of \eqref{MFEqu} satisfy
\begin{equation}\label{cotacontinuidad}
\mathcal{W}_1\left(\mu_t, \nu_t\right) \le
r(t) \mathcal{W}_1(\mu_{in},\nu_{in}),
\end{equation}
where the distance $\mathcal{W}$ is defined in \eqref{normaM}.
\end{theorem}

\noindent The proof of this result is given in section \ref{ProofThmTeoExist} below.

As before, the estimate \eqref{cotacontinuidad} allows us to consider equation
\eqref{MFEqu} as the correct limit as $N\to +\infty$ of the interacting agents system
\eqref{MainEqDelayImperfect}.

However, system \eqref{MainEqDelayImperfect} is an example of the general system
\eqref{MainEqDelay}, we must study how both  equations \eqref{MFEqu} and \eqref{FP5} are
related. For any $s\in [-\tau,0]$, denote $ev(s)$ the evaluation operator defined by
$$
\begin{array}{rccl}
  ev[s]:& E &\rightarrow & \R^d \\
   &\sigma&\mapsto & \sigma(s).
\end{array}
$$
Notice in particular that for any $\mu(0)\in \Prob(E)$ and any $s\in [-\tau,0]$,
we have $ev(s)\sharp\mu(0)\in \Prob(\R^d)$.
Moreover if we assume that $\mu(0)$ is supported in some ball $B_0(R_0)\subset E$ then
the curve $s\in [-\tau,0]\to ev(s)\sharp\mu(0)\in \Prob(\R^d)$ is continuous.

So, we have

\begin{theorem}\label{Coherent}
Fix some  $R_0>0$ and an initial distribution $\mu(0)\in \Prob(E)$  supported in the ball
$B_0(R_0)\subset E$, and denote $\mu(t)$ be the unique solution of the fixed-point equation
\eqref{FP5} as given by Theorem \ref{GeneralThm}. Then $\tilde\mu(t):=ev(0)\sharp \mu(t)\in
\Prob(\R^d)$ is the unique solution of \eqref{MFEqu} with initial condition $\tilde
\mu(s):=ev(s)\sharp\mu(0)$, $s\in [-\tau,0]$.
\end{theorem}

\noindent The proof of this result is deferred to section \ref{ProofThemCOherent} below.

\section{Conclusions and future work}

In this work we considered an interacting population of agents incorporating a delay in the
dynamic. We argued that a useful way to study this dynamic for an arbitrary number of
agents is through a delayed differential equation with random initial condition where the
vector field depends on the distribution of the solution itself. We showed well-posedness of
this equation using a fixed-point argument. Eventually we studied a particular case where
information is lost and how it is related to the general case.

\medskip

To conclude we  want to comment briefly on some possible extensions of this work.

First we could add a noise to the random delayed differential equation \eqref{MainRDDE},
and consider for instance the stochastic delayed differential equation
\begin{equation}\label{NoiseRDDE}
\left\{
\begin{array}{rl}
 dX(t) & = \F[\mu(t)](t,X(t))dt + \sigma(\mu(t))dW(t), \qquad \mu(t)=\mathcal{L}(X_t)  \\
\mu(0) & = \mathcal{L}(X_0)
\end{array}
\right.
\end{equation}
This is a delayed version of the classic McKean-Vlasov process \cite{McKean}. 
In the absence of delay, well-posedness results are known under various regularity and growh  assumptions of the coefficients (see e.g. the classical \cite{Sznitman}, the references in 
\cite{bao2020milstein},  and also Chapter 5 in \cite{ramirez2015existence} and Chapter 6 in \cite{kolokoltsov2010nonlinear}
 for a very general theory). We also mention the recent preprint
\cite{bao2020milstein} where the authors study a McKean stochastic equation, with a {\it
fixed} delay $\tau$, of the form
\begin{align*}
 dX(t) = & b(X(t),X(t-\tau),\mathcal{L}(X(t),\mathcal{L}(X(t-\tau)))dt  \\
& + \sigma(X(t),X(t-\tau),\mathcal{L}(X(t),\mathcal{L}(X(t-\tau)))dW(t).
\end{align*}

Delayed dynamic for interacting population could also be interesting from a modelling point
of view. Indeed, an agent updates its position reacting to the past trajectories of other
agents. In particular he could react to some geometric features of these trajectories thus
allowing to anticipate to some extent the future movements of others agents. This should be
interesting in the modelling of crowds or interactions predator-preys.

\bigskip

The remaining sections of the paper are devoted to the proof of Theorem \ref{GeneralThm},
Theorem \ref{TeoExist} and Theorem \ref{Coherent}.

\section{Proof of Theorem \ref{GeneralThm}}\label{SectionProofGeneralThm}

This section is devoted to the proof of Theorem \ref{GeneralThm} concerning the well-posedness
of
 \begin{equation}\label{MainRDDE2}
\left\{
\begin{split}
  X'(t) &= \F[\mu](t,X(t))=\int_\C K(X(t),\sigma)\,  d\mu(t,\sigma),\, \mu(t) = \mathcal{L}(X_t) \qquad t\ge 0, \\
  \mu(0) & = \mathcal{L}(X_0)
\end{split}
\right.
\end{equation}
where $K:\R^d \times E\to \R^d$, with $E=C([-\tau,0],\R^d)$, satisfies assumption (H) namely there exists $L>0$ such that for any $x,\tilde x\in\R^d$ and any $\sigma,\tilde\sigma\in E$,
$$ |K(x,\sigma) - K(\tilde x,\tilde\sigma)|\le L(|x-\tilde x|+|\sigma-\tilde\sigma|). $$

The proof consists in two steps.
First we want to rewrite \eqref{MainRDDE2} as a fixed-point problem.
We fix some $T>0$.
Notice first that, as a consequence of assumption (H),
for any $\nu\in C([0,T],\Prob_1(E))$, the vector field
$(t,x)\to \F[\nu](t,x)$ is continuous and Lipschitz in $x\in\R^d$ uniformly in $t\in [0,T]$.
Indeed
\begin{eqnarray*}
&& |\F[\nu](t,x)-\F[\nu](\tilde t,\tilde x)| \\
&& \qquad \le \Big| \int K(x,\sigma)\Big(\nu(t)(d\sigma)-\nu(\tilde t)(d\sigma)\Big) \Big|
+ \int |K(x,\sigma)-K(\tilde x,\sigma)|\,\nu(\tilde t)(d\sigma) \\
&& \qquad  \le L\Big(W_1(\nu(t),\nu(\tilde t))+ |x-\tilde x|\Big)
\end{eqnarray*}
where we used that the function $\sigma\to K(x,\sigma)$ is L-Lipschitz for any $x\in\R^d$.
Let $\T(s,t,x;\nu)$ be the flow of $\F[\nu](t,x)$, i. e.,
\begin{equation*}
\left\{
\begin{split}
 \frac{d}{dt} \T(s,t,x;\nu) & = \F[\nu](\T(s,t,x;\nu),t) \qquad t\in\R, \\
\T(s,s,x;\nu) & = x.
\end{split}
\right.
\end{equation*}
Then the solution of the DDE
$$  \frac{d}{dt} x(t) = \F[\nu](t,x(t)) $$
starting from $x_0\in E$ is
\begin{equation}
x(t) =
\begin{cases}
 x_0(t) \qquad t\in [-\tau,0], \\
 \T(0,t,x(0);\nu) \qquad t\ge 0.
\end{cases}
\end{equation}
%where $\T_t(x(0);\nu):=\T(0,t,x(0);\nu)$.
If the initial condition is a random variable $X_0:\Omega\to E$, then the unique pathwise solution to
 \begin{equation}\label{MainRDDE3}
  X'(t) = \F[\nu](t,X(t))
\end{equation}
is the random variable $X(t)$ defined by
\begin{equation}
X(t)(\omega)=
\begin{cases}
  X_0(t)(\omega) \qquad t\in [-\tau,0], \\
  \T(0,t,X_0(0);\nu)(\omega)) \qquad t\ge 0
\end{cases}
\end{equation}
for any $\omega\in \Omega$.

To simplify notation we introduce the following extension operator
\begin{equation}\label{definicionSt2}
\begin{array}{rccl}
  ext[\nu]:& E &\rightarrow & C([-\tau,+\infty),\R^d)\\
  &\sigma &\mapsto &\left\{ \begin{array}{ll}
    \sigma(t) &\text{ if } t\leq 0 \\
    \T(0,t,\sigma(0);\mu)&\text{ if } t> 0.
  \end{array}
  \right.
\end{array}
\end{equation}
Thus $X(t)= ext[\nu](X_0)$.
Snce $E$ is the natural state space for DDE, we also need the restriction operator
$res[t]:C([-\tau, T],\mathbb{R}^d)\to E$ defined by
$$
\begin{array}{rccl}
  res[t]:&C([-\tau, +\infty),\mathbb{R}^d)&\rightarrow & E \\
   &\sigma&\mapsto & \sigma_t.
\end{array}
$$
Notice that
$$res[t]\circ ext[\nu]:E\to E $$
and that
$$ X_t = (res[t]\circ ext[\nu])(X_0). $$
Denoting $\tilde\nu(t)$ the distribution of $X_t$ and $\mu(0)$ that of $X_0$, it follows that
\begin{equation}\label{Fixed1}
 \tilde\nu_t = (res[t]\circ ext[\nu])\sharp\mu_0.
\end{equation}
Indeed, for any $\phi:E\to \R$ measurable and bounded, we have
\begin{align*}
\int_E\phi\,d\tilde\nu(t)
 & =  \mathbb{E}[\phi(X_t)] \\
& = \mathbb{E}[\phi((res[t]\circ ext[\nu])(X_0))]
\\
& =  \int \phi\circ(res[t]\circ ext[\nu])\,d\mu(0) \\
&  =  \int \phi\,d((res[t]\circ ext[\nu])\sharp\mu(0)).
\end{align*}
In conclusion, for any $\nu\in C([0,+\infty),\Prob_1(E))$ the solution $X_t$ to
\eqref{MainRDDE3} with an initial condition distributed as $\mu_0$ has distribution
$\tilde\nu_t$ given by \eqref{Fixed1}.
We thus deduce that

\begin{step}
A stochastic process $X_t$ with distribution $\mu_t$ is solution of \eqref{MainRDDE2}
path-wise with initial condition $\mu_0$ if and only if $\mu_t$ satisfies the fixed-point
equation
\begin{equation}\label{Fixed3}
 \mu_t = (res[t]\circ ext[\mu])\sharp\mu_0 \qquad t\ge 0.
\end{equation}
\end{step}

We must therefore prove that the fixed-point equation \eqref{Fixed3} has a unique solution.
From now on we assume that the distribution $\mu_0$ of the initial condition is supported in
the $B_0(R_0)\subset E$. For a time $T>0$ to be specified later, we consider the map
$\Gamma$ defined by
\begin{equation}\label{definicionGamma}
  \begin{array}{ccl}
  \Gamma: C([0,T],\mathbb{P}_1(E))   &\rightarrow &  C([0,T],\mathbb{P}_1(E))  \\
    \mu(\cdot) &\mapsto &   (res[\cdot]\circ ext[\mu])\#\mu_0.
\end{array}
\end{equation}
We also consider the set
$$ \A_T:= \Big\{\mu\in C([0,T],\mathbb{P}_1(E)):\, \mu(0)=\mu_{\inp},\, \text{ and }
\text{supp}\,\mu(t)\subset B_0(2R_0) \, \forall\,t\in [0,T]  \Big\}. $$
We endow $\Prob_1(E)$ with the $W_1$ distance and $C([0,T],\Prob_1(E))$ with the associated $\sup$ distance
$\max_{t\in [0,T]} W_1(\mu(t),\nu(t))$ for $\mu,\nu\in C([0,T],\Prob_1(E))$.
Since $(\Prob_1(E),W_1)$ is complete, $C([0,T],\Prob_1(E))$ is also complete.
Endowed with the same distance,  $\A_T$ is therefore complete as a closed subspace of $C([0,T],\mathbb{P}_1(E))$.

We will apply the classical Banach fixed point to $\Gamma$  and $\A_T$ for  $T$ small
enough.

\begin{step}\label{Step2FP}
There exists $T>0$ depending only on $L$ (as defined in assumption (H)) such that
$ \Gamma(\A_T)\subset \A_T$ and $\Gamma$ is  a contraction in $\A_T$ in the sense that  there is a   $C\in [0,1)$ such that
\begin{equation*}%\label{Contraction2}
 \max_{t\in[0, T]} W_1(\Gamma(\mu)(t),\Gamma(\nu)(t))
 \le C \max_{t\in[0, T]} W_1\left(\mu(t), \nu(t)\right).
 \end{equation*}
\end{step}

It follows that $\Gamma$ has a unique fixed point in $\A_T$. We can then repeat the argument in
$[T,2T], [2T,3T],...$ to obtain the unique $\mu$ solving \eqref{Fixed3}.

Before beginning the proof of Step \ref{Step2FP}, we state an easy lemma concerning the flow
$\T(s,t,x;\mu)$ of the  vector field $(t,x)\to \F[\mu](t,x)$ for $\mu\in\A_T$.

\begin{lem}\label{PropFlow}
There exists $C>0$ depending only on $R_0$ and $K$ such that for any $x,y\in \Rd$, any $\mu,\nu\in \A_T$, and any $t\in [0,T]$, there holds
\begin{eqnarray}
&& |\T(t,x;\mu)|\le  |x|e^{Ct} + (e^{Ct}-1)(1+2R_0), \label{EqE1} \\
&& |\T(t,x;\mu)-\T(t,y;\mu)|\le e^{Lt} |x-y|, \label{EqE2} \\
&& |\T(t,x;\mu)-\T(t,x;\nu)| \nonumber \\
&&   \hspace{1cm}\le L\int_0^t \left( W_1(\mu(h),\nu(h))
+   L\int_0^h W_1(\mu(r),\nu(r))\,dr\, e^{L(t-h)} \right)\,dh. \label{EqE3}
\end{eqnarray}
\end{lem}

\begin{proof}
In view of assumption (H) there exists $C>0$ depending only on $K$ such that
$|K(x,\sigma)|\le C(1+|x|+|\sigma|)$ for any $x\in\R^d$ and $\sigma\in E$.
Since $\mu(t)$ is supported in $B_0(R_0)\subset E$ for any $t\in [0,T]$ we obtain
$$ \F[\mu](t,x)\le C\int_E (1+|x|+|\sigma|)\,\mu(t)(d\sigma)
\le C(1+|x|+2R_0).  $$
From this we obtain
\begin{equation*}
\begin{array}{rl}
 |\T(t,x;\mu)|
& \le |x|+\int_0^t |\F[\mu](\T(h,x;\mu),h)|\,dh  \\
& \le |x| +  C(1+2R_0)t +   C\int_0^t |\T(h,x;\mu)|\,dh.
\end{array}
\end{equation*}
We deduce \eqref{EqE1}  applying   Gr\"onwall's lemma \ref{Gronwall1}.

To prove \eqref{EqE2} we first notice that
\begin{equation}\label{LipF}
 |\F[\mu](t,x)-\F[\mu](t,y)|
\le \int_E |K(x,\sigma)-K(y,\sigma)|\,\mu(t)(d\sigma)
\le L|x-y|.
\end{equation}
Thus
\begin{eqnarray*}
|\T(t,x;\mu)-\T(t,y;\mu)|
& \le & |x-y|+\int_0^t |\F[\mu](h,\T(h,x;\mu))-\F[\mu](h,\T(h,y;\mu))| \,dh \\
& \le & |x-y|+L\int_0^t |\T(h,x;\mu)-\T(h,y;\mu)| \,dh.
\end{eqnarray*}
We deduce \eqref{EqE2}  applying  Gr\"onwall's lemma \ref{Gronwall1}.

We eventually prove \eqref{EqE3}. First for any $h\in [0,T]$ and $x\in\R^d$,
\begin{eqnarray*}
\left|\F[\mu](h,x)-\F[\nu](h,x)\right|
\le \Big| \int_E K(x,\sigma)\,(\mu(h)-\nu(h))(d\sigma)\Big|
\le LW_1(\mu(h),\nu(h))
\end{eqnarray*}
where we used that $K(x,\sigma)$ is $L$-Lipschitz in $\sigma$ for any $x\in\R^d$.
Thus using \eqref{LipF} we obtain
\begin{align*}
\left|\F[\mu](h,x)-\F[\nu](h,y)\right|
& \le \left|\F[\mu](h,x)-\F[\mu](h,y)\right|  + \left|\F[\mu](h,y)-\F[\nu](h,y)\right| \\
& \le L|x-y| + LW_1(\mu(h),\nu(h)).
\end{align*}
Thus
\begin{align*}
  | \T(t,x;\mu) - \T (t,x;\nu) |
& \le \int_0^t |\F[\mu](h,\T(h,x;\mu)) - \F[\nu](h,\T(h,x;\nu))|\,dh \\
& \le L\int_0^t W_1(\mu(h),\nu(h))\,dh +   L\int_0^t | \T(h,x;\mu) - \T (h,x;\nu) |  \,dh.
\end{align*}
Applying   Gr\"onwall's lemma gives \eqref{EqE3}.
\end{proof}

We are now in position to prove Step \ref{Step2FP}:

\begin{proof}[Proof of Step \ref{Step2FP}. ]
We first prove that $\Gamma(\A_T)\subset \A_T$. Fix some $\mu \in \A_T$.
Since $\mu_{\inp}$ is a probability measure it follows from the definition \eqref{pushforward} of the push-forward  that $\Gamma (\mu)(t) $ is a probability  measure for all $t\geq 0$.
Also because $res[0] \circ ext[\mu] $ is the identity in $E$  we have $\Gamma(\mu)(0)=\mu_{\inp}$.

For any $\sigma\in E$ such that  $\|\sigma\|\leq R_0$ we have by definition of $ext[\mu]$ that
for any $t\in [0,T]$,
\begin{align*}
  \|res[t]\circ ext[\mu](\sigma)\|
\leq \max_{h\in[-\tau,T]}|ext[\mu](\sigma)(h)|
   \leq \max_{h\in[0,T]}\{\left|\T_h(\sigma(0);\mu)\right|,\|\sigma\|\}.
\end{align*}
Using \eqref{EqE1}  we deduce that
\begin{align*}
  \|res[t]\circ ext[\mu](\sigma)\|
  &\leq \max \left\{ R_0e^{CT}+(e^{CT}-1)(1+2R_0),R_0\right\}
\end{align*}
which is $\le 2R_0$ for $T$ small enough depending only on $C$ i. e. on $K$. As a
consequence for any $\varphi: E\to\R$ supported outside the ball $  B_0(2R_0)$, we have
$\varphi(res[t]\circ ext[\mu](\sigma))=0$ for any
  $\|\sigma\|\le  R_0$. Since  $\mu_{\inp}$ is supported in $B_0(R_0)$ we obtain
$$\int_E\varphi(\sigma)\,d\Gamma(\mu)(\sigma,t)
=\int_E \varphi(res[t]\circ ext[\mu](\sigma))\,d\mu_{\inp}(\sigma) =0.$$ Thus,
$\text{supp}\,\Gamma(\mu)(t)\subset B_0(2R_0)$ for any $t\in [0,T]$.

We eventually prove that $\Gamma(\mu)(t)$ is continuous in $t$ for the distance $W_1$.
Since $\text{supp}\,\Gamma(\mu)(t)\subset B_0(2R_0)$, it is enough to prove that
$$ \lim_{s\to t} \int \phi \,d(\Gamma(\mu)(s)-\Gamma(\mu)(t))=0 $$
for any $\phi:E\to \E$ bounded continuous. Fix such a $\phi$. Then
$$ \int_E \phi \,d(\Gamma(\mu)(s)-\Gamma(\mu)(t))
= \int_E \phi(res[s]\circ ext[\mu](\sigma)) - \phi(res[t]\circ ext[\mu](\sigma))
 \,d\mu(0)(\sigma).
$$
We can pass to the limit $s\to $ by applying the Dominated Convergence Theorem using that
$res[s]\circ ext[\mu](\sigma)$ is continuous in $t$.

\medskip

We now verify that $\Gamma$ is a strict contraction for $T$ small enough.
Let  $\tau_0=\min\{\tau,t\}$.
For any  $\mu,\,\nu\in \A_T$ we have using   \eqref{EqE3} that
\begin{equation}\label{Estim10}
\begin{array}{rl}
  & \|res[t]\circ ext[\mu](\sigma)- res[t]\circ ext[\nu](\sigma)\| \\
  &=\max_{s\in[-\tau_0,0]}|\T_{t+s}(\sigma(0);\mu)-\T_{t+s}(\sigma(0);\nu)| \\
 & \leq \max_{s\in[-\tau_0,0]}
\left|L\int_0^{t+s} W_1(\mu( h),\nu( h)) \, dh \right. \\
&\hspace{2.4cm}
 \left. +L^2\int_0^{t+s}\int_0^h W_1(\mu(r),\nu(r))\, dr \, e^{L(t+s-h)}\, dh\right|\\
 & \leq L\int_0^{t} W_1(\mu( h),\nu( h)) \, dh
+L^2\int_0^{t}\int_0^h W_1(\mu(r),\nu(r))\, dr \, e^{L(t-h)}\, dh.
\end{array}
\end{equation}
Then, for every $\varphi:E\to \Rd$ 1-Lipschitz and any $t\in [0,T]$, we have using
\eqref{Estim10} that
\begin{align*}
 \left| \int_E \varphi(\sigma) \,\right.& d(\Gamma(\mu)(t)-\Gamma(\nu)(t))(\sigma)\Big| \\
 &=\int_E |\varphi(res[t]\circ ext[\mu](\sigma)) -
\varphi(res[t]\circ ext[\nu](\sigma))| \, d\mu_{\inp}(\sigma) \\
 &=\int_E \|(res[t]\circ ext[\mu](\sigma)) - (res[t]\circ ext[\nu](\sigma))\| \,
 d\mu_{\inp}(\sigma) \\
& \le \max\{L,L^2\}C_T\max_{h\in [0,T]} W_1(\mu( h),\nu( h)),
 \end{align*}
where $C_T\to 0$ as $T\to 0$.
Taking  the supremum  over all such  $\varphi$ we deduce
$$ \max_{t\in [0,T]} W_1(\Gamma(\mu)(t),\Gamma(\nu)(t))
\le \max\{L,L^2\} C_T\max_{h\in [0,T]} W_1(\mu( h),\nu( h)). $$ We thus take $T$ small
enough so that $\max\{L,L^2\}C_T<1$.
\end{proof}

We eventually prove that the unique fixed point of $\Gamma$ depends continuously on the initial condition.

\begin{step}
Given two initial condition $\mu_{ino},\nu_{inp}\in \Prob(E)$ supported in some ball $B_0(R_0)$,
the corresponding fixed-point $\mu$ and $\nu$ solution for $t\ge 0$ of
$$  \mu_t = (res[t]\circ ext[\mu])\sharp\mu_{inp} \qquad  \text{and} \qquad
\nu_t = (res[t]\circ ext[\nu])\sharp\nu_{inp}  $$
satisfy
\begin{equation}\label{ContDep}
  W_1(\mu( t),\nu(t))\leq
r(t) W_1(\mu_{\inp},\nu_{\inp}).
\end{equation}
for some continuous increasing function $r:[0,+\infty)\to [1,\infty)$ such that $r(0)=1$.
\end{step}

\begin{proof}
Observe first that for any $\mu\in \A_T$ and any $\sigma,\,\tilde\sigma\in E$, we have using \eqref{EqE2} that
\begin{align*}
&   \|res[t]\circ ext[\mu] (\sigma)- res[t]\circ ext[\mu] (\tilde\sigma) \|\\
  &\qquad \leq  \max\left\{\max_{h\in[0, t]}|\T_h(\sigma(0);\mu)-T_h(\tilde\sigma(0);\mu)|,\|\sigma-\tilde\sigma\|\right\}\\
  &\qquad \leq  \max\left\{\max_{h\in[0, t]}
        e^{L h}\,|\sigma(0)-\tilde\sigma(0)|,\|\sigma-\tilde\sigma\|\right\}\\
 &\qquad  \leq  e^{L t}\,\|\sigma-\tilde\sigma\|
\end{align*}
Therefore,   $(res[t]\circ ext[\mu])$  is $e^{L t}$-Lipschitz.

Then, for every $\varphi:\C\to \Rd$ 1-Lipschitz we have using \eqref{Estim10} that
\begin{align*}
& \Big| \int_E \varphi(\sigma) \,d(\mu( t)-\nu(t))(\sigma)\Big| \\
 &\qquad=\left| \int_E \varphi(res[t]\circ ext[\mu](\sigma)) \, d\mu_{\inp}(\sigma)
-\int_E \varphi(res[t]\circ ext[\nu](\sigma)) \, d\nu_{\inp}(\sigma)\right| \\
 &\qquad\leq  \Big|\int_E \varphi(res[t]\circ ext[\mu](\sigma)) \, d(\mu_{\inp}-\nu_{\inp})(\sigma)
\Big| \\
&\qquad \hspace{1cm} + \int_E |\varphi(res[t]\circ ext[\mu](\sigma)) -  \varphi(res[t]\circ ext[\nu](\sigma))|
\,d\nu_{\inp}(\sigma).
\end{align*}
Since  $(res[t]\circ ext[\mu])$  is $e^{L t}$-Lipschitz, the first term is bounded by
$e^{L t}\, W_1(\mu_{\inp},\nu_{\inp})$. We bound the 2nd term using \eqref{Estim10} by
\begin{eqnarray*}
&& \int_E |(res[t]\circ ext[\mu](\sigma)) - (res[t]\circ ext[\nu](\sigma))|
\,d\nu_{\inp}(\sigma)  \\
& & \le L\int_0^{t} W_1(\mu( h),\nu( h)) \, dh
+L^2\int_0^{t}\int_0^h W_1(\mu(r),\nu(r))\, dr \, e^{L(t-h)}\, dh  \\
&& \le   L\int_0^{t} W_1(\mu( r),\nu( r)) \, dr\left(1+L\int_0^t\, e^{L(t-h)}\, dh\right)
\end{eqnarray*}
We thus obtain
\begin{align*}
& \Big| \int_E \varphi(\sigma) \,d(\mu( t)-\nu(t))(\sigma)\Big|
\le   e^{L t} W_1(\mu_{\inp},\nu_{\inp})   + Le^{Lt}\int_0^{t} W_1(\mu( r),\nu( r)) \, dr.
\end{align*}
Taking the supremum over all such $\phi$ gives
\begin{align*}
& W_1(\mu(t),\nu(t))
\le   e^{Lt} W_1(\mu_{\inp},\nu_{\inp})   + L e^{Lt}\int_0^{t} W_1(\mu( r),\nu( r)) \, dr.
\end{align*}
We deduce \eqref{ContDep} applying Gronwall's Lemma \ref{Gronwall1}.
\end{proof}

\section{Proof of Theorem \ref{TeoExist}. } \label{ProofThmTeoExist}

The proof of Theorem \ref{TeoExist} is very similar to the proof of Theorem \ref{GeneralThm} so we will only sketch it.

We fix $R_0>0$ and an initial condition $\mu_{in}\in C([-\tau,0],\Prob(\R^d))$  such that    $\mu_{in}(s)(B_0(R_0))=1$ for any $s\in [-\tau,0]$.

Recall that for any  $\mu\in C([-\tau,T],\Prob(\R^d))$,
the unique solution in $C([-\tau,T],\Prob(\R^d))$ of
$$ \p_t\nu + \text{div}(\tilde\F[\mu]\nu)=0,\qquad t\ge 0 $$
with initial condition $\nu_0=\mu_{in}$ and where $\tilde F[\mu]$ is defined in \eqref{VF3},
is given by $\nu=\Gamma{\mu}$ where
$\Gamma: C([-\tau,T],\mathbb{P}_1(\Rd)) \rightarrow  C([-\tau,T],\mathbb{P}_1(\Rd))$
is defined by
$$ \Gamma[\mu](t) = \left\{
       \begin{array}{ll}
      		\mu_{in}(t)  &\text{ if } t\leq 0 ,\\
      		\tilde\T (t,\cdot;\mu)\#\mu_{in}(0)  &\text{ if } t\ge 0.
    \end{array}
    \right.
$$
Here $\tilde\T (t,x;\mu)$ is the flow of $\tilde F[\mu]$.
We must therefore prove that $\Gamma$ has a unique fixed point.

We  look for such a fixed point in the set $\A_T$ defined by
$$ \A_T = \{\mu\in C([-\tau,T],\mathbb{P}(\Rd)):\, \mu_0=\mu_{in},\,
\text{supp}\, \mu(t)\subset B_0(2R_0), \, \forall\,t\in[0,T]\}. $$ We endow
$C([-\tau,T],\mathbb{P}(\Rd))$ with the distance
$\max_{s\in[-\tau,T]}W_1\left(\mu(s),\nu(s)\right)$ that makes it complete. With the
induced distance, $\A_T$ is closed and thus complete.

Because we assumed that $\tilde K$ satisfies assumption (H'),
properties \eqref{EqE1} and \eqref{EqE2} stated in Lemma \ref{PropFlow} hold with
$\tilde T(t,x;\mu)$ in place of $T(t,x;\mu)$. Concerning  \eqref{EqE3}, it must be replaced by
\begin{equation}\label{Eq20}
\begin{split}
 |\tilde\T(t,x;\mu)-\tilde\T(t,x;\nu)|
\le Le^{Lt} \int_0^t \mathcal{W}_1(\mu_s,\nu_s)\,ds
 \end{split}
\end{equation}
for any $\mu,\nu\in\A_T$, $t\ge 0$, and $x\in\R^d$, and where
$\mathcal{W}_1(\mu_s,\nu_s)=\max_{h\in [-\tau,0]} W_1(\mu(s+h),\nu(s+h))$. Indeed,
\begin{align*}
|\tilde\F[\mu](t,x)-\tilde\F[\nu](t,\bar x)|
& \le  |\tilde\F[\mu](t,x)-\tilde\F[\mu](t,\bar x)|
+  |\tilde\F[\mu](t,\bar x)-\tilde\F[\nu](t,\bar x)| \\
& \le L|x-\bar x|  \\
& \qquad + \int_{-\tau}^0 \int_{\R^d} \tilde K(\bar x,y)(\mu(t+s)(dy)-\nu(t+s)(dy))\,d\rho(s).
\end{align*}
Since $K(\bar x,.)$ is $L$-Lipschitz, the second term can be bounded by
$$ L \int_{-\tau}^0 W_1(\mu(t+s)(dy),\nu(t+s)(dy))\,d\rho(s)
\le L\mathcal{W}_1(\mu_t,\nu_t).  $$
Thus
$$ |\tilde\F[\mu](t,x)-\tilde\F[\nu](t,\bar x)|
\le L|x-\bar x|+L\mathcal{W}_1(\mu_t,\nu_t).  $$
We then obtain
\begin{align*}
|\tilde\T(t,x;\mu)-\tilde\T(t,x;\nu)|
& \le \int_0^t |\tilde\F[\mu](s,\tilde\T(s,x;\mu))-\tilde\F[\nu](s,\tilde\T(t,x;\nu))| \,ds  \\
& \le L\int_0^t \mathcal{W}_1(\mu_s,\nu_s)\,ds
+ L \int_0^t |\tilde\T(s,x;\mu)-\tilde\T(s,x;\nu)|\,ds.
\end{align*}
Gronwall's inequality then gives \eqref{Eq20}.

We can then prove, as in the proof of Theorem \ref{GeneralThm}, that for a small enough $T$
depending only on $L$ (defined in assumption (H')), we have $\Gamma(\A_T)\subset \A_T$
and $\Gamma$ a strict contraction, i. e., there is a constant $C\in [0,1)$ such that
\begin{equation*}
 \max_{t\in[0, T]} W_1\left(\Gamma(\mu)(t),\Gamma(\nu)(t)\right)
 \le C \max_{t\in[0, T]} W_1(\mu(t), \nu(t)) \qquad \mu,\nu\in \A_T.
 \end{equation*}
We deduce the existence of a unique fixed point of $\Gamma$ in $\A_T$. Iterating this
argument gives the existence and uniqueness statement of Theorem \ref{TeoExist}.

To finish the proof we have to prove the continuous dependence of the solution with respect
to the initial data. Fix another initial condition $\nu_{in}$ and denote $\nu$ the
corresponding solution. For every 1-Lipschitz $\varphi:\Rd\to \R$ we have for $t\ge 0$ that
\begin{align*}
 & \int_\Rd\varphi(x)\, (\mu(t)(dx)-\nu(t)(dx))  \\
 &\qquad =\int_\Rd\varphi(\tilde\T(t,x;\mu))\,\mu_{in}(0)(dx)
  -\int_\Rd\varphi(\tilde\T(t,x;\nu))\,\nu_{in}(0)(dx)  \\
&\qquad = \int_\Rd\varphi(\tilde\T(t,x;\mu))\,(\mu_{in}(0)(dx)-\nu_{in}(0)(dx)) \\
& \qquad \qquad+ \int_\Rd\varphi(\tilde\T(t,x;\mu))-\varphi(\tilde\T(t,x;\nu))\,\nu_{in}(0)(dx)   \\
&\qquad=: A+B.
\end{align*}
We bound $A$ by
$$
A  \le  Lip(\varphi(\tilde\T(t,\cdot;\mu)))W_1(\mu_{in}(0),\nu_{in}(0))
 \le  e^{Lt}\,W_1(\mu_{in}(0),\nu_{in}(0)).
$$
Concerning $B$ we use that $\phi$ is 1-Lipschitz and \eqref{Eq20} to write
\begin{equation*}\label{acotacionT}
\begin{split}
B & \le \int_\Rd |\tilde\T(t,x;\mu))-\tilde\T(t,x;\nu)|\,\nu_{in}(0)(dx)
\le  Le^{Lt} \int_0^t \mathcal{W}_1(\mu_s,\nu_s)\,ds.
\end{split}
\end{equation*}
Thus after taking the supremum over all 1-Lipschitz functions  $\varphi$ we obtain
\begin{align*}
W_1(\mu(t),\nu(t))\leq  e^{Lt}\left(W_1(\mu_{in}(0),\nu_{in}(0))+
L\int_0^t  \mathcal{W}_1(\mu_s,\nu_s)\,dh\right) \qquad t\ge 0.
\end{align*}
Incorporating times $t\in [-\tau,0]$, we obtain
\begin{align*}
W_1(\mu(t),\nu(t))\leq  e^{Lt}\left(\mathcal{W}_1(\mu_{in},\nu_{in})+
L\int_0^t  \mathcal{W}_1(\mu_s,\nu_s)\,dh\right) \qquad t\ge -\tau.
\end{align*}
We have in particular that for any $t\ge 0$ and  $h\in [-\tau,0]$,
\begin{align*}
W_1(\mu(t+h),\nu(t+h))\leq  e^{Lt}\left(\mathcal{W}_1(\mu_{in},\nu_{in})+
L\int_0^t  \mathcal{W}_1(\mu_s,\nu_s)\,dh\right)
\end{align*}
so that
\begin{align*}
\mathcal{W}_1(\mu_t,\nu_t)\leq  e^{Lt}\left(\mathcal{W}_1(\mu_{in},\nu_{in})+
L\int_0^t  \mathcal{W}_1(\mu_s,\nu_s)\,dh\right) \qquad t\ge 0.
\end{align*}
Gronwall's lemma yields the result.

\section{Proof of Theorem \ref{Coherent}.} \label{ProofThemCOherent}

Fix some  $R_0>0$ and an initial distribution $\mu(0)\in \Prob(E)$  supported in the ball
$B_0(R_0)\subset E$, and denote $\mu(t)$ be the unique solution of the fixed-point equation
\eqref{FP5} as given by Theorem \ref{GeneralThm}. We must prove that
$\tilde\mu(t):=ev(0)\sharp \mu(t)\in \Prob(\R^d)$
 is the unique solution of \eqref{MFEqu}
with initial condition $\tilde \mu(s):=ev(s)\sharp\mu(0)$, $s\in [-\tau,0]$.

Let us first observe that
\begin{equation}\label{Compatibility}
ev(s)\sharp \mu(t)=\tilde\mu(t+s) \qquad \text{for any $t\ge 0$ and $s\in [-\tau,0]$.}
\end{equation}
Indeed for any $\sigma\in C([-\tau,+\infty),\R^d)$, we have
$$ (ev(s)\circ res(t))(\sigma) = \sigma_t(s)=\sigma(t+s) = \sigma_{t+s}(0)
= (ev(0)\circ res(t+s))(\sigma) $$
so that $ev(s)\circ res(t)=ev(0)\circ res(t+s)$.
Recalling that $\mu(t)=(res(t)\circ ext(\mu))\sharp \mu(0)$ we obtain
\begin{align*}
ev(s)\sharp \mu(t)
 & =  (ev(s)\circ res(t)\circ ext(\mu))\sharp \mu(0) \\
& =  (ev(0)\circ res(t+s) \circ ext(\mu))\sharp \mu(0) \\
 & =  ev(0)\sharp (res(t+s) \circ ext(\mu))\sharp \mu(0) \\
& =  ev(0)\sharp \mu(t+s) \\
&  =  \tilde \mu(t+s),
\end{align*}
which is \eqref{Compatibility}.

It follows that the vector field $\F[\mu]$ and $\tilde \F[\tilde \mu]$ defined in
\eqref{VF} and \eqref{VF3} coincide:
$$ \F[\mu](t,x) =  \tilde \F[\tilde \mu](t,x) \qquad \text{for any $t\ge 0$ and $x\in\R^d$.}
$$
Indeed, in view of the definition \eqref{DefKImperfect} of $K$, we have
\begin{eqnarray*}
\F[\mu](t,x)
& = &  \int_E K(x,\sigma)\,\mu(t)(d\sigma)
= \int_{-\tau}^0 \Big( \int_E \tilde K(x,\sigma(s))\,\mu(t)(d\sigma)  \Big) \,d\rho(s) \\
& = &
\int_{-\tau}^0 \Big( \int_{\R^d} \tilde K(x,y)\,(ev(s)\sharp\mu(t))(dy)  \Big) \,d\rho(s).
\end{eqnarray*}
Using \eqref{Compatibility} this is
\begin{align*}
\F[\mu](t,x)
 & = \int_{-\tau}^0 \Big( \int_{\R^d} \tilde K(x,y)\,\tilde\mu(t+s)(dy)  \Big) \,d\rho(s) \\
 & =  \tilde \F[\tilde \mu](t,x).
\end{align*}

We can now prove that $\tilde\mu(t):=ev(0)\sharp \mu(t)$ is the unique solution
of \eqref{MFEqu} with initial condition $ev(t)\sharp \mu(0)$, $t\in [-\tau,0]$.
 According to the proof of Theorem \ref{TeoExist} (see the definition of
 $\Gamma$) this is equivalent to proving that
\begin{equation*}
\tilde\mu(t) =
\begin{cases}
ev(t)\sharp \mu(0), \quad t\in [-\tau,0], \\
\tilde T(t,\cdot;\tilde \mu)\sharp \tilde \mu(0) \quad t\ge 0,
\end{cases}
\end{equation*}
where $\tilde T(t,x;\tilde \mu)$ is the flow of the vector field $\tilde F[\tilde\mu]$.
The equality for $t\in [-\tau,0]$ follows from \eqref{Compatibility}.
For $t\ge 0$, using  the definition \eqref{definicionSt} of $ext(\mu)$,
we have for any $\sigma\in E$,
\begin{align*}
(ev(0)\circ res(t)\circ ext(\mu))(\sigma)
=  ext(\mu)(\sigma)(t) = \T(t,\sigma(0);\mu)
\end{align*}
where  $\T(t,x;\mu)$ is the flow of the vector field $\F[\mu]$. Thus,
$$ ev(0)\circ res(t)\circ ext(\mu) = \T(t,\cdot;\mu)\circ ev(0). $$
It follows that
\begin{align*}
 \tilde \mu(t)
 & =  ev(0)\sharp \mu(t) \\ 
 & =  (ev(0)\circ res(t)\circ ext(\mu))\sharp \mu(0)\\ 
& = \T(t,\cdot;\mu)\sharp (ev(0)\sharp \mu(0)) \\
&  =  \T(t,\cdot;\mu)\sharp \tilde\mu(0).
\end{align*}
Since $\F[\mu]=\tilde\F[\tilde\mu]$, we have $T(t,x;\mu)=\tilde T(t,x;\tilde \mu)$. Thus,
$\tilde \mu(t) = \tilde\T(t,\cdot;\tilde \mu)\sharp \tilde\mu(0)$ for any $t\ge 0$, which
completes the proof.

%\addcontentsline{toc}{chapter}{Bibliografy}
\bibliography{bibliodelay}

\begin{thebibliography}{10}

\bibitem{AS}
Azmy~S. Ackleh and Nicolas Saintier.
\newblock Diffusive limit to a selection-mutation equation with small mutation
  formulated on the space of measures.
\newblock {\em Discrete and Continuous Dynamical Systems Serie B.}
\newblock in press.

\bibitem{Bellomo}
Giacomo Albi, Nicola Bellomo, Luisa Fermo, Seung~Yeal Ha, J.~Kim, Lorenzo
  Pareschi, David Poyato, and Juan Soler.
\newblock Vehicular traffic, crowds, and swarms: From kinetic theory and
  multiscale methods to applications and research perspectives.
\newblock {\em Mathematical Models and Methods in Applied Sciences},
  29(10):1901--2005, 2019.

\bibitem{AFMS}
Luigo Ambrosio, M~Fornassier, Marco Morandotti, and Giuseppe Savare.
\newblock Spatially inhomogeneous evolutionary games.
\newblock https://arxiv.org/pdf/1805.04027.pdf.

\bibitem{bao2020milstein}
Jianhai Bao, Christoph Reisinger, Panpan Ren, and Wolfgang Stockinger.
\newblock Milstein schemes for delay {M}c{K}ean equations and interacting
  particle systems.
\newblock {\em arXiv preprint arXiv:2005.01165}, 2020.

\bibitem{ants2}
Emmanuel Boissard, Pierre Degond, and Sebastien Motsch.
\newblock Trail formation based on directed pheromone deposition.
\newblock {\em Journal of mathematical biology}, 66(6):1267--1301, 2013.

\bibitem{CCC}
Jose~A. Canizo, Jose~A. Carrillo, and Silvia Cuadrado.
\newblock Measure solutions for some models in population dynamics.
\newblock {\em Acta Appl Math}, 123:141--156, 2013.

\bibitem{CCR}
Jose~A. Canizo, Jose~A. Carrillo, and Jesus Rosado.
\newblock A well-posedness theory in measures for some kinetic models of
  collective motion.
\newblock {\em Mathematical Models and Methods in Applied Sciences},
  21(3):515--539, 2011.

\bibitem{choicucker}
Young-Pil Choi and Jan Haskovec.
\newblock Cucker-{S}male model with normalized communication weights and time
  delay.
\newblock {\em Kinet. Relat. Models}, 10(4):1011--1033, 2017.

\bibitem{EBA}
Emiliano Cristiani, Benedetto Piccoli, and Andrea Tosin.
\newblock {\em Multiscale Modeling of Pedestrian Dynamics}, volume~12.
\newblock Springer International Publishing, 2014.

\bibitem{dafermos}
Constantine~M Dafermos.
\newblock Asymptotic stability in viscoelasticity.
\newblock {\em Archive for rational mechanics and analysis}, 37(4):297--308,
  1970.

\bibitem{FabGP}
Mauro Fabrizio, Claudio Giorgi, and Vittorino Pata.
\newblock A new approach to equations with memory.
\newblock {\em Archive for rational mechanics and analysis}, 198(1):189--232,
  2010.

\bibitem{ants}
Marco~A Fontelos and Avner Friedman.
\newblock A pde model for the dynamics of trail formation by ants.
\newblock {\em Journal of Mathematical Analysis and Applications},
  425(1):1--19, 2015.

\bibitem{golsedynamics}
Fran\c{c}ois Golse.
\newblock On the dynamics of large particle systems in the mean field limit.
\newblock In {\em Macroscopic and large scale phenomena: coarse graining, mean
  field limits and ergodicity}, volume~3 of {\em Lect. Notes Appl. Math.
  Mech.}, pages 1--144. Springer, [Cham], 2016.

\bibitem{hale}
Jack~K. Hale.
\newblock Theory of functional differential equations.
\newblock {\em Applied Mathematical Sciences}, 3, 1977.

\bibitem{CCGU}
Carrillo J.A., R.M. Colombo, P~Gwiazda, and A~Ulikowska.
\newblock Structured populations, cell growth and measure valued balance laws.
\newblock {\em Journal of differential equations}, 252:3245--3277, 2012.

\bibitem{jafari2001no}
Amir Jafari, Amy Greenwald, David Gondek, and Gunes Ercal.
\newblock On no-regret learning, fictitious play, and nash equilibrium.
\newblock In {\em ICML}, volume~1, pages 226--233, 2001.

\bibitem{koessler2000common}
Fr{\'e}d{\'e}ric Koessler.
\newblock Common knowledge and interactive behaviors: A survey.
\newblock {\em European Journal of Economic and Social Systems},
  14(3):271--308, 2000.

\bibitem{kolokoltsov2010nonlinear}
Vassili~N Kolokoltsov.
\newblock {\em Nonlinear {M}arkov processes and kinetic equations}, volume 182.
\newblock Cambridge University Press, 2010.

\bibitem{McKean}
Henry McKean.
\newblock A class of {M}arkov processes associated with nonlinear parabolic
  equations.
\newblock {\em Proceedings of the National Academy of Sciences of the USA},
  56:1907--1911, 1966.

\bibitem{munz2008delay}
Ulrich Munz, Antonis Papachristodoulou, and Frank Allgower.
\newblock Delay-dependent rendezvous and flocking of large scale multi-agent
  systems with communication delays.
\newblock In {\em 2008 47th IEEE Conference on Decision and Control}, pages
  2038--2043. IEEE, 2008.

\bibitem{PT}
Lorenzo Pareschi and Giuseppe Toscani.
\newblock {\em Interacting Multiagent Systems}.
\newblock Oxford University Press, 2014.

\bibitem{PPS}
Lucia Pedraza, Juan~Pablo Pinasco, and N~Saintier.
\newblock Measure-valued opinion dynamics.
\newblock {\em Mathematical Models and Methods in Applied Sciences},
  30(2):225--260, 2020.

\bibitem{PPSS}
Mayte Perez-Llanos, Juan~Pablo Pinasco, N.~Saintier, and Analia Silva.
\newblock Opinion formation models with heterogeneous persuasion and zealotry.
\newblock {\em SIAM Journal on Mathematical Analysis}, 50(5):4812--4837, 2018.

\bibitem{pinasco2018game}
Juan~Pablo Pinasco, Mauro~Rodriguez Cartabia, and Nicolas Saintier.
\newblock A game theoretic model of wealth distribution.
\newblock {\em Dynamic Games and Applications}, 8(4):874--890, 2018.

\bibitem{PRCS}
Juan~Pablo Pinasco, Mauro Rodriguez~Cartabia, and Nicolas Saintier.
\newblock Evolutionary game theory in mixed strategies: from microscopic
  interactions to kinetic equations.
\newblock
  \url{https://mate.dm.uba.ar/~nsaintie/juegos_Cartabia_Pinasco_Saintier.pdf}.

\bibitem{PSV}
Juan~Pablo Pinasco, Nicolas Saintier, and Federico Vazquez.
\newblock Role of voting intention in public opinion polarization.
\newblock {\em Physical Review E}, 101, 012101, 2020.

\bibitem{paperopinion}
Mayte Pérez-Llanos, Juan~P. Pinasco, Nicolas Saintier, and Analía Silva.
\newblock Opinion formation models with heterogeneous persuasion and zealotry.
\newblock {\em SIAM Journal on Mathematical Analysis}, 50(5):4812--4837, 2018.

\bibitem{ramirez2015existence}
Dialid~Santiago Ram{\'\i}rez.
\newblock {\em On the Existence of a Certain Class of Nonlinear Stochastic
  Processes}.
\newblock PhD thesis, University of Warwick, 2015.

\bibitem{SPV}
Nicolas Saintier, Juan~Pablo Pinasco, and Federico Vazquez.
\newblock A model for a phase transition between political mono-polarization
  and bi-polarization.
\newblock {\em Chaos: An Interdisciplinary Journal of Nonlinear Science}, in
  press.

\bibitem{smith2011introduction}
Hal~L Smith.
\newblock {\em An introduction to delay differential equations with
  applications to the life sciences}, volume~57.
\newblock Springer New York, 2011.

\bibitem{Sznitman}
Alain-Sol Sznitman.
\newblock {\em Topics in Propagation of Chaos - Ecole d'\'et\'e de
  probabilit\'es de Saint-Flour XIX - 1989}, volume 1464.
\newblock Springer-Verlag, 1991.

\bibitem{toscani2006kinetic}
Giuseppe Toscani et~al.
\newblock Kinetic models of opinion formation.
\newblock {\em Communications in mathematical sciences}, 4(3):481--496, 2006.

\bibitem{villanioptimal}
C{\'e}dric Villani.
\newblock {\em Topics in optimal transportation}.
\newblock Number~58. American Mathematical Soc., 2003.

\end{thebibliography}
\bibliographystyle{plain}

\end{document}